\documentclass[11pt]{amsart}
\usepackage{amssymb}
\usepackage{amsfonts}

\usepackage{xy}
\xyoption{all}

\usepackage{setspace}

\newtheorem{theorem}{Theorem}[section]
\newtheorem{corollary}[theorem]{Corollary}
\newtheorem{lemma}[theorem]{Lemma}
\newtheorem{proposition}[theorem]{Proposition}
\newtheorem{definition}[theorem]{Definition}

\newtheorem{remark}[theorem]{Remark}
\newtheorem{openquestion}[theorem]{Open Question}
\newcommand{\Hom}{{\rm Hom}}

\newcommand{\End}{{\rm End}}
\newcommand{\Ext}{{\rm Ext}}

\def\bea{\begin{eqnarray*}}
\def\eea{\end{eqnarray*}}

\newcommand{\Aa}{\mathcal{A}}
\newcommand{\Bb}{\mathcal{B}}

\newcommand{\Ff}{\mathcal{F}}

\newcommand{\Ll}{\mathcal{L}}
\newcommand{\Mm}{\mathcal{M}}

\newcommand{\Ss}{\mathcal{S}}

\def\NN{{\mathbb N}}
\def\CC{{\mathbb C}}
\def\KK{{\mathbb K}}

\begin{document}

\sloppy

\title[Complete Path Algebras and Rational Modules]{Complete Path Algebras and Rational Modules}


\dedicatory{\it Dedicated to my mentor, esteemed professor and dear friend, prof. Constantin N\u ast\u asescu, on the occasion his 70th birthday, and to Professor Toma Albu, with deepest admiration.}

\begin{abstract}
We study rational modules over complete path and monomial algebras, and the problem of when rational modules over the dual $C^*$ of a coalgebra $C$ are closed under extensions, equivalently, when is the functor $Rat$ a torsion functor. We show that coreflexivity, closure under extensions of finite dimensional rational modules and of arbitrary modules are Morita invariant, and that they are preserved by subcoalgebras. We obtain new large classes of examples of coalgebras with torsion functor, coming from monomial coalgebras, and answer some questions in the literature.
\end{abstract}

\author{M.C. Iovanov${}^{1,2}$}
\address{${}^1$University of Bucharest, Facultatea de Matematica\\ Str.
Academiei 14, Bucharest 1, RO-010014, Romania}
\address{${}^2$University of Iowa\\
MacLean Hall, Iowa City, Iowa 52246, USA}
\address{e-mail: iovanov@usc.edu, yovanov@gmail.com} 


\date{}

\maketitle

\section{Introduction and Preliminaries}

Let $Q$ be a quiver, i.e. an oriented graph; loops are allowed and infinitely many arrows between two vertices are also possible. The path algebra (or quiver algebra) $\KK[Q]$ of $Q$ over a field $\KK$ is the $\KK$-vector space spanned by the oriented paths in $Q$, and the multiplication defined by $p*q=pq$ - the concatenation of $p$ and $q$ if the starting point of $q$ is the same as the ending point of $p$, and $p*q=0$ otherwise (note that the opposite multiplication of this is sometimes taken as the path algebra multiplication). Another combinatorial object of importance is the path coalgebra $\KK Q$, which as a vector space is the same as the path algebra, but the comultiplication $\Delta$ is given by $\Delta(r)=\sum\limits_{r=pq}p\otimes q$ and the counit is $\varepsilon(p)=\delta_{0,|p|}$, where $|p|$ denotes the length of $p$, and $\delta$ is the Kroeneker symbol. In \cite{DIN1}, connections between the two are established. One interesting interpretation of the path coalgebra is that the locally nilpotent representations of $Q$ are precisely the comodules over the path coalgebra $\KK Q$. Recall that a representation of $Q$, is equivalently, a module over the path algebra of $Q$; a $\KK[Q]$-module $M$ is called locally nilpotent if every element of $M$ is annihilated by a cofinite monomial ideal, i.e. an ideal $I$ spanned by paths, which are monomials in ``arrows'', and such that $I$ has finite codimension (see also \cite{chin}). 

A third important algebraic object associated sometimes to quivers is the so called complete path algebra. By definition, the complete path algebra of $Q$ is the set of sequences $(\alpha_p)_p$ indexed by oriented paths in $Q$, and with multiplication defined by:
$$(\alpha_p)_p*(\beta_q)_q=(\sum\limits_{r=pq}\alpha_p\beta_q)_r$$
One way to understand this is via the language of coalgebras. If $C=\KK Q$ is the quiver (path) coalgebra of $Q$ and $A=\KK[Q]$ is the path algebra, then $A$ embeds naturally in $C^*$. Via this embedding, $C^*$ can be understood as the completion of $A$ with respect to the topology in which a basis of neighborhoods of $0$ in $A$ is given by cofinite monomial ideals (see \cite{DIN1}). Then $A$ is dense in $C^*$, which is a complete (pseudocompact) algebra, where on $C^*$ one considers the usual finite topology of orthogonals of finite dimensional subspaces of $C$. It is not difficult to see that $C^*$ is then isomorphic to the above mentioned complete path algebra of $Q$ as an algebra, and so this justifies the name.

The category of locally nilpotent representations of $Q$ is then precisely that of rational modules over the complete path algebra $C^*$. The name comes from algebraic geometry; given an algebraic group (scheme) $G$ over $\CC$, the category of rational representations of $G$ is equivalent to that of comodules over the coalgebra (Hopf algebra) of representative functions of $G$, which coincides to the algebra of functions on $G$. Also, in positive characteristic, the (rational) representations of $G$ are the comodules over the hyperalgebra of $G$. More generally, given an algebra-coalgebra pairing $(B,D)$, i.e. a morphism $\varphi:B\rightarrow D^*$, there is a natural notion of rational modules $(B,D)$-modules (see \cite{HR,Rad}), and the category of such rational modules is again a category of comodules over a coalgebra ($D/\varphi(B)^\perp$); this category embeds naturally in the category of $B$-modules. The above situations in algebraic geometry are of this type.

In particular, given a coalgebra $C$, the category $Rat({}_{C^*}\Mm)$ of rational $C^*$-modules (or equivalently, of $C$-comodules) is a subcategory of ${}_{C^*}\Mm$ of modules over $C^*$. This category is closed under coproducts, subobjects and quotients, and provides an interesting example of pre-torsion theory. Many investigations in coalgebras are in fact connected to the relation between rational modules and arbitrary modules over $C^*$. For example, one type of problem studied in this setting is when does the rational part of every module or of every finitely generated module over $C^*$ split off \cite{C,I3,I4,I5,NT}.

One problem of a particular interest is deciding when is this subcategory of rational modules closed under extensions. This has been considered before by many authors \cite{C0, CNO, HR, I2, NT1, L, L2, Rad, Sh, TT}. In general, given an abelian or Grothendieck category $\Aa$ and a closed subcategory $\Bb$ of $A$, one can consider the trace functor $T$ with respect to $\Bb$, which is defined as $T(M)=$the sum of all subobjects of $M$ which belong to $\Bb$. This is a pre-torsion functor or pre-radical. It is well known and easy to see that the subcategory $\Bb$ is closed under extensions, or {\it localizing}, if and only if $T$ is a radical, (or a torsion functor), i.e. $T(M/T(M))=0$ for all $M$ in $\Aa$. Such examples of localizing categories are, of course, important also from the point of view of general noncommutative Gabriel localization theory. One main goal of this paper is to study this problem for the above setting of complete path algebras, and their rational (or locally nilpotent) modules. In fact, we will consider more general algebras, such as ``complete monomial algebras'', which are similar to the complete path algebras, but the indexing set of their elements  is a certain fixed suitable subset of the set all of paths of $Q$.

In view of the above, we recall the following.

\begin{definition}
Let $C$ be a coalgebra. We say that $C$ has a left torsion Rat-functor, if the functor $Rat: {}_{C^*}{\Mm}\rightarrow Rat({}_{C^*}\Mm)=\Mm^C$ is a torsion functor, equivalently, the category of rational left $C^*$-modules is closed under extensions in ${}_{C^*}\Mm$. 
\end{definition}

For such a situation we will also use the alternate terminology ``left rational ($C^*$-)modules are closed under extensions''. We refer the reader to \cite[page 2]{I2} for a brief history of this problem. Many important classes of coalgebras are known to have a (left) torsion rational functor. For example, such are the right semiperfect coalgebras (see \cite{L,NT1}), or if $C$ is such that closed cofinite (equivalently, open) left ideals of $C^*$ (in the finite topology) are finitely generated (in this case $C^*$ is said to be left $\Ff$-Noetherian). Since all the classes of coalgebras with torsion rational functor seemed to be $\Ff$-Noetherian, the authors asked in \cite{CNO} if this is perhaps an equivalent characterization; this motivated the research of \cite{TT}, where a counterexample was produced. 
However, in certain situations this property can be characterized equivalently by $\Ff$-Noetherianity; for example, if the coradical $C_0$ is finite dimensional, then $C$ has a left rational torsion functor if and only if $C^*$ is left $\Ff$-Noetherian, and equivalently, all the terms of the coradical filtration are finite dimensional (see \cite{C0,CNO,HR}). In this situation, the notion becomes left-right symmetric. However, the problem of finding an equivalent characterization for when rational modules are closed under extensions remained open. Also mainly motivated by the interesting developments in \cite{CNO}, such an equivalent characterization was obtained in \cite[Theorem 3.7]{I2}, in the form of a topological condition on the open ideals of $C^*$ and a homological $\Ext$ condition. 

\vspace{.2cm}

{\bf [ \cite{I2}, Theorem 3.7]} {\it A coalgebra $C$ has a left Rational torsion functor (equivalently, left rational $C^*$-modules are closed under extensions in ${}_{C^*}\Mm$), if and only if the finite dimensional rational $C^*$-modules are closed under extensions in ${}_{C^*}\Mm$ (equivalently, equivalently, the set of open ideals of $C^*$ is closed under the product of ideals), and for every simple right $C$-comodule $T$ and every injective indecomposable right $C$-comodule $E$, $\Ext^1(T,E)=0$, where $\Ext^1=\Ext^1_{C^*}$ is taken in the category of left $C^*$-modules.}

\vspace{.2cm}

Using this result in part, a generalization to arbitrary coalgebras of the above equivalent characterizations for the situation when $C_0$ is finite dimensional is given in \cite[Theorem 4.8]{I2}. This allowed one to get simple procedures to generate coalgebras with a torsion Rat functor, as well as coalgebras for which Rat is not torsion, and at the same time, new examples of non-$\Ff$-noetherian algebras $C^*$ with rational torsion functor. Some results in \cite{I2} have the advantage that they allow one to consider examples coming from quivers, and so use combinatorial intuition. However, the conditions of the equivalent characterization of \cite[Theorem 3.7]{I2} are not easy to check, and so these results in their own lead to new open questions. One such question is whether the existence of a torsion left rational functor implies the fact that the right Rat functor is also torsion (\cite[Question 2]{I2}). Another is whether the above homological $\Ext$ condition of \cite[Theorem 3.7]{I2} is in fact always true (\cite[Question 3]{I2}), or perhaps it can be eliminated from \cite[Theorem 3.7]{I2}, which would mean the torsion of the Rat functor is a left right symmetric question (\cite[Question 1]{I2}). 

Here, we aim to provide new large classes examples of coalgebras with rational torsion functor, and perhaps, offer new insights to what is needed to answer these questions. We aim to establish combinatorial conditions on quivers $Q$ or on certain sets of paths of quivers, that can be easily checked and which can precisely tell if the path coalgebra of $Q$ (or more generally, some subcoalgebra) has a torsion Rat functor or not. These will at least show how a potential example of a left and not right Rat-torsion coalgebra should look like (or rather, how it cannot look like). We obtain three such conditions on quiver coalgebras, and more generally, on monomial coalgebras (or path subcoalgebras): Theorem \ref{t.2.fin}, Theorem \ref{t.2.bound} and Theorem \ref{t.1}. In many situations, we will be show that if finite dimensional rational modules are closed under extensions, then arbitrary rational (left or right) modules are also closed under extensions, thus giving a partial answer for large classes of coalgebras, to Questions 1 and 2 from \cite{I2} mentioned above. For example, we show that for a monomial coalgebra $H$ (i.e. a subcoalgebra of the a coalgebra which has a basis of paths), if for every two vertices $v,w$ the length of paths in $H$ between $v$ and $w$ is bounded by some number depending on $v$ and $w$, then the two conditions mentioned above are equivalent, i.e., closure under extensions of finite dimensional rational modules is enough to guarantee closure under extensions of all rational modules. In particular, using also results of \cite{DIN1}, this shows that monomial coalgebras for which between every two vertices there are only finitely many paths (in the monomial generating set) have left and right rational torsion functors. 

We also provide a complete answer to Question 3 from \cite{I2} mentioned above. We consider the ``thick arrow quiver'', by which we understand the quiver with two vertices $a,b$ and countably many arrows, all starting at $a$ and ending at $b$. For this quiver, we are able to precisely compute the above homological $\Ext$ spaces, in terms of the cardinality of the base field $\KK$, and show they are non-zero. The main question still remains, and can be appended with a new question: if the finite dimensional rational $C^*$-modules are closed under extensions in ${}_{C^*}\Mm$ for the coalgebra $C$, or more generally, if $C$ is locally finite, does it follow that $\Ext^1_{C^*}(T,E)=0$ for simple comodules $T$ and injective indecomposable comodules $E$?

In fact, it was noted in \cite{I2} that the problem of when is Rat a torsion functor is closely related to the notion of coreflexivity for a coalgebra $C$. Recall that $C$ is coreflexive if and only if the canonical embedding $C\hookrightarrow (C^*)^0$ is an isomorphism. There are many equivalent characterizations of coreflexive coalgebras; for example, a coalgebra is coreflexive if and only if every finite dimensional module over $C^*$ is rational, or, equivalently, every cofinite ideal of $C^*$ is closed in the finite topology on $C^*$ (see \cite{HR, Rad, RAD}). A coalgebra $C$ is coreflexive if and only if $C_0$ is coreflexive and the finite dimensional rational modules are closed under extensions in ${}_{C^*}\Mm$  (equivalently, the set of open ideals of $C^*$ is closed under the product of ideals; see \cite[Proposition 2.1 \& Corollary 2.2]{I2}). This leads us to and motivates the introduction of the following 

\begin{definition}
We say that a coalgebra $C$ is {\it directly coreflexive} if the finite dimensional rational left (equivalently, right) $C^*$-modules are closed under extensions in ${}_{C^*}\Mm$, equivalently, the set of closed cofinite (=open) ideals of $C^*$ is closed under products (see \cite[Proposition 1.9]{I2}).
\end{definition}

It is easy to see that this is a left-right symmetric notion. Moreover, by the results of \cite{HR} (see \cite[3.7.5 Theorem]{HR}), the coradical of $C_0$ is coreflexive for any coalgebra over an infinite field. Thus, coreflexivity and direct coreflexivity are very close (almost equivalent) notions, modulo some set theory considerations, which are not of consequence in any usual example.

Among the general results of this paper, we also prove three interesting and perhaps unexpected facts: direct coreflexivity, coreflexivity and the closure under extensions for rational $C^*$-modules (inside ${}_{C^*}\Mm$) are Morita invariant notions, in the sense that if $D$ is a coalgebra whose category of comodules is equivalent to that of $C$, then these properties are preserved from $C$ to $D$. This is perhaps a bit surprising, since the definitions of these notions do not appear to be categorical at first. We also show that these properties are preserved by subcoalgebras, and these two methods allow us to extend the results on quiver and monomial coalgebras to general coalgebras, via the Gabriel (Ext) quiver of the coalgebra. 



We refer the reader to \cite{DNR} and the recent monograph \cite{RAD} for basic theory of coalgebras and their comodules. We fix some notation. Let $V$ be a vector space. For $X$ a subspace of $V$, we denote $X^\perp_{V^*}=\{f\in V^*| f\vert_X =0\}$, and for $Y\subseteq V^*$, we denote $Y^\perp_{V}=\{v\in V| f(v)=0,\forall f\in Y\}$. We write simply $X^\perp$ and $Y^\perp$ when there is no danger of confusion. We consider the finite topology on $V^*$, which is a linear topology with a basis of neighborhoods for $0$ consisting of subspaces of the form $W^\perp$, with $W$ a finite dimensional subspace of $V$. We recall that for a subspace $X$ of $V$, we have $(X^\perp)^\perp=X$, and for a subspace $Y$ of $V^*$, one has that the closure of $Y$ is $\overline{Y}=(Y^\perp)^\perp$.

\section{The Thick Arrow Quiver}

In this section, we aim to answer one of the above questions connected to the problem of closure under extensions of rational $C^*$-modules. In \cite{I2}, given the characterization of this property via a topological condition and a homological condition, the question if this second homological condition is perhaps superfluous and is always true was raised. This is motivated by the fact that all the known examples of coalgebras which are (directly) coreflexive in fact have a torsion Rational functor. Specifically, \cite[Question 3]{I2} asks whether for every simple left $C$-comodule $S$ and every injective indecomposable left $C$-comodule $E$, $\Ext^1_{C^*}(S,E)=0$. We note that this is true in many situations; for example, \cite[Theorem 4.8]{I2} shows that if a very simple condition on the $\Ext$ quiver of $C$ is satisfied, then the answer to this question is true. Namely, if for every simple left $C$-comodule $S$, the second term $L_1(E(S))$ of the coradical filtration of the injective hull $E(S)$ of $S$ is finite dimensional ($\dim(L_1(E(S)))<\infty$), then $\Ext^1_{C^*}(S,E)=0$ for every simple left comodule $S$ and every indecomposable left $C$-comodule $E$. In particular, this happens if the injective hulls of simple right comodules are artinian. A similar  result was already known before by \cite[Theorem 3.2]{NT}, which stated that a coalgebra $C$ is artinian as a left $C^*$-module if and only if it is injective as a right $C^*$-module. We show that the answer to this question is negative in general, by looking at the path coalgebra of the ``thick arrow quiver'', the quiver having two vertices $a,b$ and infinitely many arrows $x_n$ between $a$ and $b$:
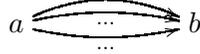
$$\xymatrix{
a \ar@/^2ex/[rr] \ar@/^1ex/[rr]_\dots \ar@/^-1ex/[rr]_{\dots} & & b
}$$ 
Let $C$ be the quiver coalgebra of this quiver over a field $\KK$. Then $C$ has as basis the set $\{a,b\}\cup\{x_n|n\geq 1\}$, where $x_n$ represent the arrows. We prefer to work with left modules, and right comodules. Let $S=\KK a$; this is a coalgebra and a left and right simple comodule. Let $E$ be its injective hull as right comodule; it has a basis $\{a\}\cup\{x_n|n\geq 1\}$. Let $T=\KK b$; $T$ is a simple and injective right $C$-comodule, and a simple left $C$-comodule. We show that $\Ext^1_{C^*}(T,E)\neq 0$. In fact, we are able to precisely compute this $\Ext$ space. Let $P$ be the projective cover of $T$ as a left $C^*$-module; it is the dual of the injective hull $E_l(T)$ of $T$ as a left comodule (e.g. by \cite[Lemma 1.4]{I}). Note that $E_l(T)$ has basis $\{b;x_n, n\geq 1\}$, and so we have an exact sequence of left comodules
$$0\longrightarrow T\longrightarrow E_l(T) \longrightarrow S^{(\NN)}\longrightarrow 0$$
(we use $S^{(\NN)}$ for the coproduct power) which yields the exact sequence of left $C^*$-modules
$$0\longrightarrow S^{\NN}\longrightarrow P=E_l(T)^*\longrightarrow T\longrightarrow 0$$
Applying the long exact sequence of homology we obtain the exact sequence, where $\Hom=\Hom_{C^*}$, $\Ext^1=\Ext^1_{C^*}$
$$0\rightarrow\Hom(T,S)\rightarrow \Hom(P,S)\rightarrow \Hom(S^{\NN},S)\rightarrow \Ext^1(T,S)\rightarrow \Ext^1(P,S)\rightarrow...$$
Since $\Ext^1(P,S)=0$, $\Hom(P,S)=0$ ($S\not\cong T$), we obtain $\Ext^1(T,S)\cong \Hom(S^\NN,S)$. One easily notes that $\Hom(S^\NN,S)=\Hom_\KK(\KK^\NN,\KK)=(\KK^\NN)^*$. But, it is well known (see for example \cite[Chapter IX, Theorem 2]{J}) that for a (left) vector space $V$ over a skewfield $D$, the dual vector space $V^*$ has dimension $\dim_D(V^*)=|D|^{\dim(V)}$, where $|D|$ is the cardinality of $D$. Hence, $\dim(\KK^\NN)^*=|\KK|^{\dim(\KK^\NN)}=|\KK|^{|\KK|^\NN}$ since $\KK^\NN=(\KK^{(\NN)})^*$. Therefore, $\dim(\Ext^1(T,S))=|\KK|^{|\KK|^\NN}$\\
Now, consider the short exact sequence of right comodules and left $C^*$-modules
$$0\longrightarrow S\longrightarrow E\longrightarrow T^{(\NN)}\longrightarrow 0$$
Applying the long exact sequence in homology again (as $C^*$-modules) we get
$$0\rightarrow \Hom(T,S)\rightarrow \Hom(T,E)\rightarrow \Hom(T,T^{(\NN)})\rightarrow \Ext^1(T,S)\rightarrow \Ext^1(T,E)\rightarrow \Ext^1(T,T^{(\NN)})\dots$$
Now, since $T^\NN$ is semisimple of length (dimension) at least $\NN$, we have that $T^{(\NN)}$ embeds in $T^\NN$. Note that obviously $T^\NN$ is semisimple: since it is annihilated by $(\KK b)^\perp$, its properties can established when regarded as a module over the algebra $(\KK b)^*\cong \KK$. Thus $T^{(\NN)}$ is a direct summand in $T^\NN$. Also, $T$ is an injective right comodule, and since it is finite dimensional, it is injective also as a left $C^*$-module (see \cite[Section 2.4]{DNR}). Therefore, $T^\NN$ is an injective left $C^*$-module, and so $T^{(\NN)}$ is an injective left $C^*$-module too. This shows that $\Ext^1(T,T^{(\NN)})=0$, and since $\Hom(T,E)=0$, the above exact sequence yields that the following sequence is exact
$$0\longrightarrow \Hom(T,T^{(\NN)})\longrightarrow \Ext^1(T,S)\longrightarrow \Ext^1(T,E)\longrightarrow 0$$
Regarded as vector spaces, since obviously $\Hom(T,T^{(\NN)})=\Hom(T,T)^{(\NN)}$, we get 
$$\dim(\Ext^1(T,S))=\dim(\Ext^1(T,E))+|\NN|=\max\{\dim(\Ext^1(T,E)),\aleph_0\}$$
But $\dim(\Ext^1(T,S))=|\KK|^{|\KK|^\NN}>2^{2^\NN}>|\NN|$, and so we get
$$\dim(\Ext^1(T,S))=\dim(\Ext^1(T,E))=|\KK|^{|\KK|^\NN}$$ 
This is obviously bigger than $2^{2^{\aleph_0}}$, and so there is a set of extensions larger than that of the continuum (and also an equally large set of isomorphism of modules of length 2 which are extensions of $T$ by $S$). In particular $\Ext^1(T,E)\neq 0$. 

The coalgebra $C$ of the thick arrow quiver is not locally finite, i.e. $\Ext^1_C(T,S)$ is infinite dimensional (here $\Ext^1_C$ means extensions in the category of comodules). Thus, one question still remains, which one may be even tempted to state as a conjecture:


\begin{openquestion}\label{q1}
If $C$ is a locally finite coalgebra, $S$ is a simple left $C$-comodule and $E$ is an indecomposable injective left $C$-comodule, does it follow that $\Ext^1_{C^*}(S,E)=0$?
\end{openquestion}


Note that by \cite[Proposition 1.3, Proposition 2.1]{I2} and their proofs, a directly coreflexive coalgebra is locally finite.  Hence, in particular, another more particular form of this question from \cite{I2} remains as well:


\begin{openquestion}\label{q2}
If $C$ is a directly coreflexive coalgebra, $S$ is a simple left $C$-comodule and $E$ is an indecomposable injective left $C$-comodule, does it follow that $\Ext^1_{C^*}(S,E)=0$? Equivalently, does every directly coreflexive coalgebra have a torsion rational functor? 
\end{openquestion}

\section{Morita Invariance and Categorical Properties}

For an arbitrary algebra $A$ and a left $A$-module $M$, denote by $Lf(M)$ the largest locally finite submodule of $M$, equivalently, the sum of all finite dimensional submodules of $M$. This is a pre-torsion functor on the category of left $A$-modules. Given a vector space $M$, we view $M$ as a subspace of $(M^*)^*$ via the canonical embedding given by evaluation, i.e. $\iota_M:M\hookrightarrow (M^*)^*$ has $\iota_M(f)(m)=f(m)$. The following proposition essentially states that a coalgebra is coreflexive if and only if its finitely cogenerated comodules satisfy a coreflexivity property, in the sense that they are isomorphic to a certain double dual.

\begin{proposition}
A coalgebra $C$ is coreflexive if and only if $Lf((M^*)^*)=M$ for every right finitely cogenerated $C$-comodule $M$. Moreover, when $C$ is coreflexive, then $Rat((M^*)^*)=M$ for every such $M$.
\end{proposition}
\begin{proof}
Assume first $C$ is coreflexive. Let $M\subseteq C^n$, and let $p:(C^*)^n\rightarrow M^*$ be the canonical dual map. Take $\alpha\in Lf((M^*)^*)$; then $U=C^*\cdot \alpha$ is finite dimensional. Therefore, $Y=U^\perp_{M^*}=\{m^*\in M^* \mid\, \beta(m^*)=0,\,\forall \beta\in U\}$ has finite codimension in $M^*$. Since $C$ is coreflexive, $M^*/Y$ is rational. We claim that $Y$ is closed in $M^*$, i.e. $Y=X^\perp_{M^*}$ for some $X\subseteq M$. The argument is similar to that of page 4 of \cite{I2}, for example. Let $W$ be the coalgebra of coefficients of $M^*/Y$; then $W^\perp_{C^*}(M^*/Y)=0$. This implies that $(W^n\cap M)^\perp_{M^*} \subseteq Y$. Indeed, note that $W^\perp_{C^*}$ is an ideal of $C^*$, that $W^\perp_{C^*}\cdot(C^*)^n=(W^n)^\perp_{(C^*)^n}$, and that $p(W^\perp_{C^*} \cdot (C^*)^n)=W^\perp_{C^*} M^*\subseteq Y$. But it is easy to see using definitions that $p\left((W^n)^\perp_{(C^*)^n}\right)=(W^n\cap M)^\perp_{M^*}$, and so we get that $(W^n\cap M)^\perp_{M^*}\subseteq Y$. Moreover, $W^n\cap M$ is finite dimensional, so $(W^n\cap M)^\perp_{M^*}$ is cofinite open. This shows that $Y$ is open, and that $Y=X^\perp_{M^*}$, for some subcomodule $X$ of $M$, by properties of the finite topology. \\
Now, we have that $\dim(X)=\dim(X^*)=\dim(M^*/X^\perp_{M^*})=\dim(M^*/Y)=\dim(M^*/U^\perp_{M^*})=\dim(U)$. By construction, it is easy to see that the canonical map $\iota_M:M\hookrightarrow (M^*)^*$ takes $X$ into $U$, and since they have the same dimension, $\iota_M(X)=U$. This shows that $U\subseteq M$ (modulo canonical identification). \\
Conversely, assume the condition holds for every right comodule $M$. Let $I$ be a right cofinite ideal of $C^*$, and let $U=I^\perp_{(C^*)^*}\subseteq (C^*)^*$. Then $U$ is a finite dimensional submodule of $(C^*)^*$, i.e. $U\subseteq Lf(C^*)^*$. Applying the hypothesis for $M=C$, one sees that there is a finite dimensional $X\subset C$ such that $\iota_C(X)=U$. Now, it is a direct application of the definition that $X^\perp_{C^*}=\iota(X)^\perp_{C^*}(=U^\perp_{C^*})$. Now, using the preliminary remarks, note that $\left(I^\perp_{(C^*)^*} \right)^\perp_{C^*}=I$, and therefore, it follows that $X^\perp_{C^*}=U^\perp_{C^*}=I$. This shows that $I$ is closed in the finite topology of $C^*$, and so $C$ is coreflexive by one of the equivalent definitions of coreflexivity. \\
The last statement is obvious, since locally finite modules over $C^*$ are rational, again by another  equivalent definition of coreflexivity.
\end{proof}

We will use the theory of basic coalgebras \cite{CM}. Recall that a coalgebra $C$ is basic if every simple subcoalgebra is simple as a $C$-comodule (equivalently, simple $C$ comodules have multiplicity one in $C$). For every coalgebra $C$, let $\Ss$ be a set of representatives for the simple left $C$-comodules, and $C=\bigoplus\limits_{S\in \Ss}S^{n(S)}$. Recall from \cite{Rad2} that for an idempotent $e$ of $C^*$, the subspace $eCe$ of $C^*$ has a coalgebra structure given by $ece\longmapsto ec_1e\otimes ec_2e$ and counit equal to the restriction of the counit $\varepsilon$ of $C$ to $eCe$. Moreover, we have that the dual ring is $(eCe)^*\cong eC^*e$. If $e$ is the idempotent of $C^*$ which equals $\varepsilon$ exactly on one of the summands isomorphic to $E(S)$ for each $S$, then the coalgebra $eCe$ is basic, and this is the basic coalgebra of $C$. We call such an $e$ a basic idempotent. This coalgebra can also be understood as obtained from a (categorical) cohom coalgebra construction. The coalgebra $C$ and its basic coalgebra $eCe$ are Morita-Takeuchi equivalent, i.e. $\Mm^C$ and $\Mm^{eCe}$ are equivalent categories. The equivalence is given by the functor $\Mm^C \ni M\longmapsto eM\in \Mm^{eCe}$. This functor is in fact defined on ${}_{C^*}\Mm$ with values in ${}_{eC^*e}\Mm$, but it may not be an equivalence between the full categories of modules. We note that two basic coalgebras are Morita equivalent if and only if they are isomorphic. Moreover, two coalgebras are Morita equivalent if and only if they have isomorphic basic coalgebras. We will need an extension of this equivalence of categories of comodules over $C$ and its basic coalgebra.

\begin{definition}
Let $C$ be a coalgebra. We denote by $\Ll(C)$ the localizing subcategory of ${}_{C^*}\Mm$ generated by the simple $C$-comodules, that is, the subcategory of left $C^*$-modules which is closed under coproducts, quotients, subobjects, and extensions and contains the simple rational $C^*$-modules.
\end{definition}

Obviously, $\Mm^C$ is a full subcategory of $\Ll(C)$, and $\Ll(C)$ consists of the semiartinian $C^*$-modules all of whose simple subquotients are rational. Let us denote $H=e\cdot ()=\Hom_{eC^*e}(C^*e,-):{}_{C^*}\Mm\rightarrow {}_{eC^*e}\Mm$, so $H(M)=eM$, and let $L=C^*e\otimes_{eC^*e}(-)$ be its left adjoint. Then $H\circ L=\rm{Id}$ - the identity functor on $eC^*e$-modules. For a semiartinian module $M$, we denote by $lw(M)$ the Loewy length of $M$.

\begin{lemma}\label{l.eq}
The functor $H$ restricts to an equivalence $H:\Ll(C^*)\rightarrow \Ll(eC^*e)$, where $e$ is a basic idempotent.
\end{lemma}
\begin{proof}
Let us fist note that $H$ is faithful. For this, we show first that $H(M)\neq 0$ for $M\neq 0$. Indeed, if  $M\neq 0$, there is a simple rational submodule $S$ of $M$, and since $eS\neq 0$, we have $eS\subset H(M)=eM\neq 0$. Therefore, if $0\neq f:M\rightarrow N$, then $H({\rm Im}(f))\neq 0$, so the image of $H(f)$ contains $H({\rm Im}(f))=e{\rm Im}(f)$, and is nonzero. Obviously, $H$ is full since $H\circ L={\rm Id}$.\\
Now, note that $L(\Ll(eC^*e))\subset \Ll(C^*)$. To see this, we proceed by transfinite induction on the Loewy length of $M$, and show that moreover $lw(L(M))\leq lw(M)$. For a semisimple rational $eC^*e$-module $M$, obviously $L(M)$ is semisimple rational, since it is standard to see that for a simple rational $eC^*e$-module $S$, $C^*e\otimes_{eC^*e}S$ is simple rational. Assume the statement holds for $\alpha$. If $\alpha=\beta+1$ is a successor, then consider the exact sequence $0\rightarrow M_\beta\rightarrow M_\alpha\rightarrow M_\alpha/M_\beta\rightarrow 0$. This yields the exact sequence $L(M_\beta)\rightarrow L(M_\alpha)\rightarrow L(M_\alpha/M_\beta)\rightarrow 0$, and from here we easily conclude that $L(M_\alpha)$ is semiartinian with rational simple subquotients (i.e. $L(M_\alpha)$ belongs to $\Ll(C^*)$) and $L(M_\alpha)$ has Loewy length less than $\beta+1$. If $\alpha$ is a limit ordinal, let $M=\bigcup\limits_{\beta<\alpha}M_\beta$, with $lw(M_\beta)=\beta$, then consider the epimorphism $\bigoplus\limits_{\beta}M_\beta\rightarrow M\rightarrow 0$. This yields an epimorphism $\bigoplus\limits_{\beta<\alpha}L(M_\beta)\rightarrow L(M)\rightarrow 0$ (since $L$ commutes with $\bigoplus$ and is right exact), so $L(M)$ is in $\Ll(C^*)$ and $lw(L(M))\leq \lim\limits_{\beta<\alpha}\beta = \alpha$. Similarly, using the fact that $H$ is (right) exact and also commutes with direct sums, one can show that $H(\Ll(C^*))\subset \Ll(eC^*e)$. \\
Now, $H$ is surjective on objects, since for $N$ in $\Ll(eC^*e)$, the object $M=L(N)$ satisfies $H(M)=N$. Therefore, $H$ induces a full and faithful functor from $\Ll(C^*)$ to $\Ll(eC^*e)$ which is ``surjective" on objects, so it is an equivalence. 
\end{proof}

\begin{proposition}
Let $C$ be a coalgebra. \\
(i) The finite dimensional rational modules are closed under extensions for $C$ if and only if the same holds for $eCe$.\\
(ii) $C$ is coreflexive if and only if $eCe$ is coreflexive.\\
(iii) $C$ has a torsion left Rat-functor if and only if its basic coalgebra $eCe$ has a torsion left Rat-functor.
\end{proposition}
\begin{proof}
(i) We note that closure under extensions in ${}_{C^*}\Mm$ for the finite rational modules over $C^*$ is equivalent to closure under extensions in $\Ll(C^*)$ for finite rational modules. Using Lemma \ref{l.eq}, the statement follows.\\
(ii) This is true since $C$ is coreflexive if and only if $C_0$ is coreflexive and finite rational modules are closed under extensions. Now, for cosemisimple coalgebras, Morita invariance for coreflexivity is easily seen to follow from the characterization of \cite[3.12 Corollary]{Rad}. Hence, using this, since $C_0$ and $(eCe)_0$ are Morita equivalent, (ii) now follows from (i). \\
(iii) To show this, by \cite[Theorem 3.7]{I2} and (i), we only need to prove that the homological condition $\Ext_{C^*}^1(T,E)=0$ is Morita invariant from $C$ to $eCe$, for $T$ simple rational and $E$ an injective comodule. This condition equivalent to the statement that every exact sequence of $C^*$-modules $0\rightarrow E\rightarrow M\rightarrow T\rightarrow 0$ is split. But for any such sequence, $M$ is obviously an object of $\Ll(C)$, so this property depends only on the category $\Ll(C)$. Therefore, by Lemma \ref{l.eq}, that the condition holds over $C$ if and only if it holds over $eCe$. This ends the proof.
\end{proof}

Using this and the fact that two coalgebras are Morita-Takeuchi equivalent if and only if they have the same basic coalgebra, we get the following

\begin{corollary}\label{c.Morita}
The following properties are Morita-Takeuchi invariant for coalgebras $C$:\\
$\bullet$ coreflexivity\\
$\bullet$ direct coreflexivity, i.e. closure under extensions for finite dimensional rational modules inside $C^*$-modules.\\
$\bullet$ existence of a left (right) Rat-functor (equivalently, closure under extensions of $Rat({}_{C^*}\Mm)$ inside ${}_{C^*}\Mm$).
\end{corollary}

It is known that a subcoalgebra of a coreflexive coalgebra is coreflexive \cite[Proposition 6.4]{T1}. We can now see that closure under extensions of finite dimensional or of arbitrary rational modules is a property also preserved by subcoalgebras.

\begin{proposition}\label{p.sub}
Let $C$ be a coalgebra, and $D$ a subcoalgebra of $C$.\\
(i) If $C$ is directly coreflexive, then so is $D$.\\ 
(ii) If left rational $C^*$-modules are closed under extensions in the category of all $C^*$-modules, equivalently, $C$ has a left Rat torsion functor, then the same holds for $D$.
\end{proposition}
\begin{proof}
Let $0\rightarrow M'\rightarrow M\rightarrow M''\rightarrow 0$ be an exact sequence of left $D^*$-modules, with $M'$ and $M''$ rational. Let $p:C^*\rightarrow D^*$ be the canonical surjective morphism (restriction-to-$D$ map). Then, regarded as $C^*$-modules, in each case (finite dimensional modules or not), we obtain that $M'$ and $M''$ are rational $C^*$-modules, and so $M$ is a rational $C^*$-module. Since $M$ is a $D^*$-module, then $M$ is annihilated by $D^\perp$. Let $\rho:M\rightarrow M\otimes C$ be the comodule map corresponding to the left rational $C^*$-module $M$. Since $D^\perp\cdot M=0$, one easily sees that, in fact, $\rho(M)\subset M\otimes D$. Therefore, $M$ is a rational left $D^*$-module: indeed, let $d^*\in D^*$ and $c^*\in C^*$ such that $p(c^*)=d^*$; then $d^*\cdot m= c^*\cdot m=\sum c^*(m_1)m_0=\sum d^*(m_1)m_0$, since $\rho(m)=\sum m_0\otimes m_1\in M\otimes D$.
\end{proof}

\section{Complete path and monomial algebras}

Let $Q$ be a quiver, that is, an oriented graph; loops, cycles and multiple arrows are allowed. For an arrow $a$ in $Q$, we denote $s(a)$ and $t(a)$ its source and target. For a path $p$ in $Q$ we write $p=qr$ if $p$ is the path obtained by following $q$ and then $r$, assuming that $r$ starts where $q$ ends. Write $|p|$ for the length of $p$. Recall that the quiver coalgebra is a coalgebra $\KK Q$ whose basis is given by the set of paths in $Q$, and comultiplication and counit are 
\begin{eqnarray*}
\Delta(p) & = & \sum\limits_{p=qr}q\otimes r \\
\epsilon(p) & = & \delta_{0,|p|}
\end{eqnarray*}
The simple comodules of $\KK Q$ correspond to the vertices of $Q$. For vertices $v,w$ in $Q$, we denote by $a(v,w)$ the cardinality of the set of arrows from $v$ to $w$. This is the dimension of the space of  $\Ext_C^1(\KK w,\KK v)$ of comodule extensions of right comodules $0\rightarrow \KK v \rightarrow M\rightarrow \KK w \rightarrow 0$; this can be infinite. For each vertex $v$, let $R(v)$, (respectively, $L(v)$) be the set of paths starting from (respectively, ending at) $v$. The span of $R(v)$ is the injective hull $E_r(v)$ of the right simple comodule $\KK v$ (\cite{chin}; see also \cite{simson}). Similarly, we use $E_l(v)$ for the injective hull of the left comodule $\KK v$. 
We can apply one of the results of \cite{I2} to obtain a general criterion for relating the coreflexivity property and the Rat torsion property for a quiver coalgebra. 

\begin{theorem}\label{t.1}
Let $Q$ be a quiver. Assume that for each vertex $v$ of $Q$, there is a number $n(v)$ and a set $X(v)$ of paths in $Q$, such that:\\
$\bullet$ (Q1) $X(v)$ consists of paths ending at $v$ (i.e. $X(v)\subseteq L(v)$), of length at most $n(v)$.\\
$\bullet$ (Q2) If $M(v)$ is the set of arrows in $L(v)$ that do not appear at the end of a path in $X(v)$, then the sets  
$S_w=\{x\in M(v) \vert s(x)=w \}$ are finite of cardinality bounded by some number $m(v)$ (finite cardinal).\\
Then the quiver coalgebra $\KK Q$ is directly coreflexive if and only if it has a left torsion Rat functor (i.e. left rational modules are closed under extensions).
\end{theorem}
\begin{proof}
By the aforementioned \cite[Theorem 3.7]{I2}, we just need to prove the only if part. Let $v$ be a vertex of $Q$, and let $X_v$ be the span of the paths in $X(v)$ and their subpaths ending at $v$. Then it is easy to check that $X_v$ is a left subcomodule of $E_l(\KK v)={\rm Span}(L(v))$, the injective hull of the left comodule $\KK v$. Also, $X_v$ has finite Loewy length (coradical filtration), by the first condition. Moreover, $E_l(\KK v)/X_v$ has socle isomorphic to $\bigoplus\limits_{x\in M(v)}(\KK s(x))$, and here the multiplicity of each simple left comodule $\KK w$ equals $|\{x\in M(v)\vert s(x)=w\}|=|S_w|\leq m(v)$. Therefore, the socle of $E_l(\KK v)/X_v$ embeds in $(\KK Q)^{(m(v))}$ by the second condition. Hence, $E_l(\KK v)/X_v$ is finitely cogenerated. As $\KK Q$ is directly coreflexive, we can apply \cite[Theorem 4.2]{I2} and conclude that left rational $(\KK Q)^*$-modules are closed under extensions.
\end{proof}

Since the direct coreflexivity and closure under extensions of rational modules are Morita invariant properties of coalgebras, we easily get the following:

\begin{corollary}\label{c.2.Q12}
Let $C$ be a coalgebra whose simple comodules have trivial endomorphism rings, i.e. $\End^C(S)=\KK$ for all $S\in\Ss$ (for example, this is true when $\KK$ is algebraically closed). If the $\Ext$ quiver $Q$ of $C$ satisfies the conditions (Q1) and (Q2) of the above theorem, and the quiver coalgebra $\KK Q$ is directly coreflexive, then $C$ has a torsion left rational functor (i.e. left rational $C^*$-modules are closed under extensions).  
\end{corollary}
\begin{proof}
Let $D$ be the basic coalgebra of $C$. Then $D$ is pointed, so it embeds in the quiver coalgebra $\KK Q$ of its $\Ext$ quiver $Q$ (see \cite{CM,M1}). By the previous Theorem, since $\KK Q$ is directly coreflexive, it has a torsion left Rat functor. Using Proposition \ref{p.sub}, we see that $D$ has a torsion left Rat functor. By Morita invariance, since $C$ is Morita-Takeuchi equivalent to $D$, it follows that $C$ has a left torsion Rat functor.
\end{proof}

The above theorem gives a large class of coalgebras having a torsion Rational functor, provided that one can test that the quiver (path) coalgebra if its Ext quiver is (directly) coreflexive. The results can likely be extended at least to coalgebras for which $C_0$ is separable, but we refrain from going into overly technical details, and leave that to the interested reader. In general, however, coreflexivity is not easily tested. \cite[Proposition 5.4]{DIN1} gives a sufficient condition for a path coalgebra of a quiver $Q$ to be coreflexive, for general enough quivers, namely, if between every two vertices of $Q$ there are only finitely many paths. We will show that this condition is also enough for the Rat functor to be a torsion functor. We show something a little more general in what follows. 

Let $Q$ be a quiver. Let $H_0$ be a set of paths in $Q$, and let $H=\overline{H_0}$ be the set of all paths in $Q$ that are subpaths of some path in $H_0$, i.e. $H$ is the set of all subpaths of paths in $H_0$. Let $C$ be the coalgebra whose basis is given by all the paths in $H=\overline{H_0}$ (it is easy to see that this is a subcoalgebra of the path coalgebra of $Q$; equivalently, one can start with a set $H$ of paths in $Q$ which is closed under taking subpaths). Such a subcoalgebra of the path coalgebra of $Q$ which a basis of paths is called a monomial coalgebra (or path subcoalgebra, see \cite{DIN2}). For each path $p\in H$, let $p^*$ denote the element of $C^*$ for which $p^*(q)=\delta_{pq}$, for each $q\in H$. Obviously, $C^*\cong\prod\limits_{p\in H}\KK p^*$ as vector spaces, so elements of $C^*$ can be written as families $(\alpha_p)_{p\in H}$. It is easy to see that under this identification, the multiplication of $C^*$ becomes

$$(\alpha_p)_{p\in H} * (\beta_q)_{q\in H} = (\sum\limits_{s=pq})\alpha_p\beta_q$$

where for each $s\in H$, the sum is taken over all factorizations $s=pq$ of the path $s$ into concatenation of subpaths. This is obviously a generalization of the formal power series algebra, which is the dual of the path coalgebra of the one-vertex-one arrow quiver 
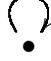
\vspace{.2cm}

$$\xymatrix{\bullet\ar@(ul,ur)[]}$$ (the divided power coalgebra). 
We can now prove the following. 

\begin{theorem}\label{t.2.bound}
With the above notations, assume that for every two vertices $v,w$ in $Q$, the length of paths in $H$ which start at $v$ and end at $w$ is bounded (by some number $N(v,w)$ depending on $v$ and $w$). Then the monomial coalgebra $C$ with basis $H$ is directly coreflexive if and only if rational left (equivalently, right) $C^*$-modules are closed under extensions.
\end{theorem}
\begin{proof}
Assume $C$ is directly coreflexive, but $Rat({}_{C^*}\Mm)$ is not closed under extensions. By \cite[Proposition 3.3]{I2}, there is an exact sequence $0\rightarrow M\rightarrow P\rightarrow S\rightarrow 0$ of left $C^*$-modules such that:\\
$\bullet$ $P$ is cyclic and not rational.\\
$\bullet$ $S=\KK b$ is simple rational.\\
$\bullet$ $M$ is rational, has infinite Loewy length, has simple socle $T=\KK a$ and $JM=M$, where $J=Jac(C^*)$.\\
Obviously, $P$ is local since $M=JM\subseteq JP\subseteq M$. Let $\varphi:E_l(\KK b)^*\rightarrow P$ be an epimorphism (which exists since $E_l(\KK b)^*$ is the projective cover of the left simple module $S$; see e.g. \cite[Lemma 1.4]{I}). Note that $E_l(\KK b)^*=C^* b^*$, and $P=\varphi(C^*b^*)=\varphi(C^*b^*\cdot b^*)=C^*b^*P$. Let $n=N(a,b)$, and let $x\in M$ be such that $x\in M_{n+1}\setminus M_n$. Since $P=C^*b^*\cdot P$, there is $f\in C^*b^*$ such that $f\cdot \varphi(b^*)=x$ ($\varphi(b^*)$ generates $P$). Then $C_n^\perp x$ has Loewy length 0 (i.e. it is simple), so $S=C_n^\perp x\neq 0$, and $a^\perp  C_n^\perp x=0$ ($a^\perp = (\KK a)^\perp$). We show however that $a^\perp \cdot C_n^\perp f=C_n^\perp f$, which will show in turn that $0=a^\perp C_n^\perp f\cdot \varphi(b^*)= C_n^\perp f\cdot \varphi(b^*)=C_n^\perp x$, which is a contradiction.\\
Since $f=fb^*$, we see that $f=(\lambda_p)_{p\in H}$ such that $\lambda_p=0$ if $p$ does not end at $b$. Recall that, $C_n^\perp$ consists of elements $(\alpha_q)_q$ such that $\alpha_q=0$ if the length of $q$ satisfies $|q|\leq n$. Then for every element $c^*=(\alpha_q)_q\in C_n^\perp$, we have $$c^*f=(\gamma_s)_{s\in H}=(\sum\limits_{qp=s,\,|q|>n,\, t(p)=b}\alpha_q\lambda_p)_{s\in H}$$ 
In the above formula, we see that $\gamma_s=0$ if $s$ does not end at $b$. Also, $\gamma_s=0$ if $|s|\leq n$. Therefore, $\gamma_s=0$ if $s$ starts at $a$, since if $s$ starts at $a$ and ends at $b$, it has $|s|\leq n$ by hypothesis. This shows that $a^*\cdot (\gamma_s)_{s\in H}=0$. Hence, $(\varepsilon-a^*) \cdot (\gamma_s)_{s\in H}=(\gamma_s)_{s\in H}$. Since $\varepsilon - a^*\in a^\perp$, this shows that $c^*f=(\varepsilon-b^*) c^*f\in a^\perp C_n^\perp f$. Therefore, $a^\perp C_n^\perp f= C_n^\perp f$. As noted above, this ends the proof. 
\end{proof}

A special type of quiver $Q$ with the above property is considered in \cite{DIN1}, namely, one for which between any two vertices of $Q$ there are only finitely many paths. It is proved in \cite[Proposition 5.4]{DIN1} that the path coalgebra $C$ of such a quiver $Q$ is coreflexive if and only if $C_0$ is coreflexive. One sees that, in fact, by the proof \cite[Theorem 5.2]{DIN1} and \cite[Proposition 5.3]{DIN1}, such a coalgebra is, in fact, directly coreflexive. That method applies to monomial coalgebras too:

\begin{theorem}\label{t.2.fin}
Let $Q$ be a quiver, and $H$ a set of paths of $Q$, closed under taking subpaths and such that between every two vertices $v,w$ of $H$ there are only finitely many paths in $H$ that start at $v$ and end at $w$. Then the monomial coalgebra $C$ (path subcoalgebra of $\KK Q$) with basis $H$ is directly coreflexive, and left (and also right) rational $C^*$-modules are closed under extensions in the category of all left (or right) $C^*$-modules.
\end{theorem}
\begin{proof}
The proof of \cite[Theorem 5.2]{DIN1} shows that a coalgebra with the property stated in that theorem is directly coreflexive. Moreover, the proof of \cite[Proposition 5.3]{DIN1}, which uses \cite[Theorem 5.2]{DIN1}, applies for monomial coalgebras too, and therefore shows that $C$ is directly coreflexive. Applying the previous Theorem \ref{t.2.bound}, we see that rational $C^*$-modules are closed under extensions.
\end{proof}

\begin{corollary}
Let $C$ be a coalgebra such that every simple rational module is endotrivial (i.e. $\End(S)=\KK$ for every such $S$; for example, when $\KK$ is algebraically closed). \\
(i) Assume that the path coalgebra of the $\Ext$ quiver $Q$ of $C$ is directly coreflexive (respectively, has rational torsion functor). Then $C$ is directly coreflexive (respectively, has rational torsion functor).\\
(ii) Assume that in the $\Ext$ quiver of $C$ there are finitely many paths between any two vertices. Then $C$ is coreflexive and has a torsion rational functor (i.e. rational modules are closed under extensions).
\end{corollary}
\begin{proof}
(i) Let $D$ be the basic coalgebra of $C$, which is pointed and embeds in $\KK Q$. Therefore, $D$ is directly coreflexive (respectively has rational torsion functor) by Proposition \ref{p.sub}, and by Morita-Takeuchi invariance of Corollary \ref{c.Morita}, the same follows for $C$ is too. \\
(ii) This follows as (i), using the above Theorem \ref{t.2.fin}.
\end{proof}

\begin{remark}
We note that the above results constitute further evidence to support an affirmative answer to Question \ref{q2}, and would motivate one to formulate it as a conjecture. At the same time, if one is to find an example of a coalgebra which is directly coreflexive and does not have a left torsion rational functor, it is natural to look for examples coming from subcoalgebras of quiver coalgebras (by Morita invariance), and Theorem \ref{t.1} gives hints on how such a counterexample to this conjecture of Question 2 should look like.
\end{remark}

\bigskip\bigskip\bigskip

\begin{center}
\sc Acknowledgment
\end{center}
The research was supported by the UEFISCDI Grant
PN-II-ID-PCE-2011-3-0635, contract no. 253/5.10.2011 of CNCSIS.
The research was also supported by  the
strategic grant POSDRU/89/1.5/S/58852, Project ``Postdoctoral
program for training scientific researchers'' cofinanced by the
European Social Fund within the Sectorial Operational Program
Human Resources Development 2007-2013.


\bigskip\bigskip\bigskip



\begin{thebibliography}{99}


\bibitem{abe}
E.~Abe, Hopf Algebras, Cambridge University Press, Cambridge,
1977.

\bibitem{CM}
W. Chin and S. Montgomery, \emph{Basic Coalgebras}, in Modular Interfaces (Riverside, CA, 1995) AMS/IP Studies in Advanced Math. vol.4, Providence RI, (1997) 41--47.

\bibitem{chin}
W.~Chin, M.~Kleiner, D.~Quinn, \textit{Almost split sequences for
comodules}, J. Algebra {\bf 249} (2002), 1--19.

\bibitem{C0}
J. Cuadra, \emph{Extensions of rational modules.}, Int. J. Math. Math. Sci. 2003, no. 69, 4363–-4371. 

\bibitem{C} J. Cuadra, \emph{When does the rational submodule split off?}, Ann. Univ. Ferrarra -Sez. VII- Sc. Mat. Vol. LI (2005), 291--298.


\bibitem{CNO} J. Cuadra, C. N\u ast\u asescu, F. Van Oystaeyen, \emph{Graded almost noetherian rings and applications to coalgebras}, J. Algebra {\bf 256} (2002) 97-–110.

\bibitem{DNR} S.~D\u asc\u alescu, C.~N\u ast\u asescu, \c S.~Raianu, Hopf algebras.
An introduction, Marcel Dekker, New York, 2001.

\bibitem{DIN1} S.~D\u asc\u alescu, M.C.~Iovanov, C.~N\u ast\u
asescu, \textit{Quiver algebras, path coalgebras, and
coreflexivity}, Pac. J. Math 262, No.1 (2013), 49--79.; preprint arXiv.

\bibitem{DIN2} 
S.~D\u asc\u alescu, M.C.~Iovanov, C.~N\u ast\u asescu\textit{Path subcoalgebras,  finiteness properties and quantum groups}, to appear, J. Noncommutative Geometry; preprint arXiv.

\bibitem{NT1}
J. G$\rm\acute{o}$mez-Torrecillas, C. N\u ast\u asescu, \emph{Quasi-co-Frobenius coalgebras}, J. Algebra 174 (1995), 909--923.


\bibitem{HR}
R.~Heyneman, D.E.~Radford, \textit{Reflexivity and Coalgebras of
Finite Type}, J. Algebra {\bf 28} (1974), 215--246.

\bibitem{I}
M.C.~Iovanov, \emph{Co-Frobenius Coalgebras}, J. Algebra 303
(2006), no. 1, 146--153.

\bibitem{I1}
M.C.~Iovanov, \textit{Generalized Frobenius Algebras and Hopf
Algebras}, preprint arXiv, to appear, Can. J. Math. (2013); online
http://dx.doi.org/10.4153/CJM-2012-060-7

\bibitem{I2}
M.C.~Iovanov, \textit{On Extensions of Rational Modules}, preprint
arxiv, to appear, Israel J. Math. 


\bibitem{I3}
M.C.~Iovanov, \emph{The Splitting Problem for Coalgebras: A Direct Approach}, Applied Categorical Structures 14 (2006) - Categorical Methods in Hopf Algebras - no. 5-6, 599--604.

\bibitem{I4}
M.C.~Iovanov, \emph{When does the rational torsion split off for finitely generated modules}, Algebr. Represent. Theory 12 (2009), no. 2-5, 287--309.

\bibitem{I5}
M.C.~Iovanov, C. N\u ast\u asescu, B. Torrecillas-Jover, \emph{The Dickson subcategory splitting conjecture for pseudocompact algebras},  J. Algebra  320  (2008),  no. 5, 2144--2155.

\bibitem{J}
N. Jacobson, {\it Lectures in Abstract Algebra. Volume 2: Linear algebra}, Springer-Verlag, 1951, Berlin-
Heidelberg-New York.

\bibitem{L}
B.I-Peng Lin, \emph{Semiperfect coalgebras}, J. Algebra {\bf 30}
(1974), 559--601.

\bibitem{L2}
B.I-P. Lin, \emph{Products of torsion theories and applications to coalgebras}, Osaka J. Math. 12 (1975) 433–-439.

\bibitem{mo} S.~Montgomery,  Hopf algebras and their actions
on rings, {\sl CBMS Reg. Conf. Series} {\bf 82}, American
Mathematical Society, Providence, 1993.

\bibitem{M1}
S. Montgomery, \emph{Indecomposable coalgebras, simple comodules, and pointed Hopf algebras}, Proc. Amer. Math. Soc. 123 (1995) 2343–-2351.

\bibitem{NT}
C. N\u ast\u asescu, B. Torrecillas, \emph{The Splitting Problem for Coalgebras}, J. Algebra 281 (2004) 144–-149.

\bibitem{RAD}
D. E. Radford, \emph{Hopf Algebras}, Volume 49 of K \& E series on knots and everything, World Scientific, 2011, 584pp.

\bibitem{R}
D.E.~Radford, \textit{On the Structure of Ideals of the Dual
Algebra of a Coalgebra}, Trans. Amer. Math. Soc. Vol. 198 (Oct
1974), 123--137.

\bibitem{Rad}
D.E.~Radford, \textit{Coreflexive coalgebras}, J. Algebra {\bf 26}
(1973), 512--535.

\bibitem{Rad2}
D. E. Radford, \emph{On the structure of pointed coalgebras}, J. Algebra {\bf 77} (1982), 1–-14.

\bibitem{Sh}
T. Shudo, \emph{A note on coalgebras and rational modules}, Hiroshima Math. J. 6 (1976) 297–-304.

\bibitem{simson}
D.~Simson, \textit{Incidence coalgebras of intervally finite posets, their integral quadratic forms and comodule categories}, Colloq. Math. {\bf 115} (2009), 259--295.



\bibitem{Sw} M.E.~Sweedler, Hopf Algebras, Benjamin, New York, 1969.


\bibitem{T1}
E.J.~Taft, \textit{Reflexivity of Algebras and Coalgebras},
American Journal of Mathematics, Vol. 94, No. 4 (Oct 1972),
1111--1130.

\bibitem{T2}
E.J.~Taft, \textit{Reflexivity of algebras and coalgebras. II.}
Comm. Algebra  {\bf 5}  (1977), no. 14, 1549--1560.

\bibitem{TT}
M.L. Teply, B. Torrecillas, \emph{Coalgebras with a radical rational functor}, J. Algebra 290 (2005), 491–-502.





\end{thebibliography}
\end{document}